\documentclass{amsart}
\usepackage{amsthm}
\newtheorem{theorem}{Theorem} 
\newtheorem{lemma}[theorem]{Lemma}
\newtheorem{definition}[theorem]{Definition}
\newtheorem{remark}[theorem]{Remark}
\newtheorem{proposition}[theorem]{Proposition}
% ...

% ...
\newcommand{\C}{\mathbb{C}}
\newcommand{\N}{\mathbb{N}}

\newcommand{\slt}{\sll(2,\C)}
\newcommand{\id}{\textrm{id}}

\newcommand{\p}[1]{\frac{\partial}{\partial #1}}
\DeclareMathOperator*{\Ker}{Ker}
\DeclareMathOperator*{\sll}{sl}

\renewcommand{\Im}{\operatorname{Im}}

\title
[Normal forms and nilpotent vector fields]
{Normal forms with exponentially small remainder and Gevrey normalization for vector fields with a nilpotent
linear part}

\usepackage{amsaddr}

\author{P. Bonckaert}
\address{Universiteit Hasselt\\ 
Agoralaan Gebouw D\\
3590 Diepenbeek\\
Belgium\\
patrick.bonckaert@uhasselt.be
}

\author{F. Verstringe}

\address{Universiteit Hasselt\\ 
Agoralaan Gebouw D\\
3590 Diepenbeek\\
Belgium\\
freek.verstringe@uhasselt.be
}

  

\begin{document}
%% Abstract 
\begin{abstract}
We explore the convergence/divergence of the normal form for a singularity of a vector field on $\C^n$ with
nilpotent linear part.  We show that a Gevrey-$\alpha$ vector field $X$ with a nilpotent linear part can be
reduced to a normal form of Gevrey-$1+\alpha$ type with the use of a Gevrey-$1+\alpha$  transformation. We also give a proof of
the existence of an optimal order to stop the normal form procedure. If one stops the normal form procedure at this order, the
remainder becomes exponentially small.
\end{abstract}

\maketitle

\section{Introduction and statement of the results} 
We consider vectors fields where the linear part at a singularity is nilpotent, with no restriction on the
dimension. This, for example, includes the case of a coupled Takens-Bogdanov system, see e.g.
\cite{MR2647461}. See \cite{MR1941477} for an introduction to the subject.

We briefly give some history of the subject. In \cite{MR0339292} the planar case $y\frac{\partial}{\partial
x}+ \ldots$ was considered and a formal normal form $(y+a(x))\frac{\partial}{\partial x}+
b(x)\frac{\partial}{\partial y}$ was derived. It was shown in \cite{MR1885678}, also in the planar case, that
an analytic vector field with nilpotent linear part $y\p{x}$ can be analytically transformed to a normal form.
Other results related to the planar case are in \cite{MR2116814}.

More recently it was shown in \cite{pre05797622} that the analytic vector fields with linear part
$y\p{x}+z\p{y}$ can be Gevrey-1 reduced to a normal form using a specific normal form procedure that is also
described later on in this article. This framework was extended in \cite{pre05797622} to the case of
quasihomogeneous vector fields. In \cite{pre05797622} it is explained what the generalization of the so called
small denominators are for non-diagonal linear vector fields (and more general quasihomogeneous vector
fields); and some results of convergence and Gevrey-1 normalization are explained (See also theorems
\ref{thm:ann} and \ref{thm:gev}).

In \cite{MR2130546} the optimal order to cut off the normal form procedure was determined. At this order the remainder becomes
exponentially small. This was done for vector fields with a semi-simple linear part with eigenvalues that satisfy $|\left<\lambda,\delta\right>-\lambda_j|\geq \gamma/|\delta|^\tau$ for each $\delta\in \N^n$, $j\in \left\{
1,\ldots, n
\right\}$ for which $\left<\lambda,\delta\right>-\lambda_j\neq 0$ and a certain $\gamma,\tau>0$; resonant
eigenvalues are allowed. A similar result was also
obtained
for vector fields with a completely degenerate linear part (i.e. all eigenvalues are zero) and at most one non-trivial Jordan block
of order 2 or 3. The difficulty starts with Jordan blocks of higher dimension or multiple Jordan blocks because it seems rather
difficult to compute the generalized small denominators directly using the method at hand. In this article we will succeed in computing
these generalized small denominators using  the $\slt$ structure that is hidden in the problem.

Somewhat before that, we have the results of \cite{MR855083} on the formal structure of the normal forms with
a nilpotent linear part, using representation theory of $\slt$. More recently \cite{MR2484141} and
\cite{MR2647461} have also made contributions to this subject in the multidimensional case, on the formal
level.

The purpose of this article is to combine both ideas : we will show how to use representation theory of $\slt$
in nilpotent cases in order to calculate the small denominators in the framework of \cite{pre05797622} and
hence obtain qualitative information on the growth of coefficients appearing inside the normal form procedure.

Let us state the main result. We say that the linear part $N$  of a vector field $X$ is nilpotent at $0$ if it
acts as a nilpotent linear operator on the space of polynomials of degree $\delta$, for each $\delta\in
\N\setminus\{0\}$. Note that this means, up to a linear change of the coordinates, that the linear part of the
vector field can be written as $N=\sum_{i=1}^{n-1}a_ix_{i+1}\p{x_i}$, for certain $a_1,\ldots, a_{n-1} \in
\C$. We say that a vector field $X=\sum_{i=1}^{n}X_i\p{x_i}$ is Gevrey-$\alpha$ if each $X_i$ is a formal
power series $\sum_{l\geq 0}\sum_{|k|=l}a_kx^k$ and there exists a $C,r>0$ such that $\sum_{|k|=l}|a_k|\leq
Cr^l(l!)^\alpha$; the sum is
taken over $k\in \N^n$, and $|k|=k_1+\ldots+k_n$.
\begin{theorem}
\label{theo:mainres}
Every formal Gevrey-$\alpha$ vector field $X=N+R$, where $N$ is a nilpotent linear part and $R$ is a part of
higher order, admits a formal Gevrey-$1+\alpha$ transformation to Gevrey-$1+\alpha$ normal form. If the
Gevrey-$\alpha$ vector field is formally linearizable, then the transformation and corresponding normal form
are Gevrey-$\alpha$.  
\end{theorem}

\begin{theorem}[\cite{MR2130546}, Corollary 1.9 p.7 and \cite{pre05797622}, Theorem 6.11 p.689]
\label{theo:mainres2}
Suppose that $X$ is an analytic vector field in a neighbourhood of $0$ of $\C^m$ that vanishes at the origin. This is
\begin{align*}
  X=L+\sum_{\delta\geq 2}f_{\delta},
\end{align*}
where $L$ is linear and nilpotent, $f_\delta$ is a bounded $\delta$-linear symmetric form and
\begin{align*}
  ||f_\delta(x_1,\ldots,x_\delta)||\leq \frac{c||x_1||\ldots ||x_\delta||}{\rho^\delta},
\end{align*}
with $c,\rho>0$ independent of $\delta$.
Suppose that $Q$ is an invertible linear transformation with norm $||Q||=\sup_{||x||=1}||Q(x)||$ for which $N=QLQ^{-1}$ is as in
proposition \ref{prop:mainprop}.
Let $\nu=\sup\frac{e^2p!}{p^{p+1/2}e^{-p}}$,
\begin{align*} 
  C&=\frac{\sqrt{m}}{\rho^2}\left\{ \left( \frac{5}{2}m+2 \right)c ||Q||^2.||Q^{-1}||^2+3\rho ||Q||.||Q^{-1}|| \right\}\\
M_0&=\frac{10}{9}\left\{ \left( \nu\sqrt{\frac{27}{8e}} \right)+(2e)^{2} \right\},
\end{align*}
$w=\frac{1}{eC}$ and $p_{\rm{opt}}=\left[ \frac{1}{eC} \right]$. There exists a polynomial coordinate transformation
$\id+u_{p_{\rm{opt}}}$
of degree at most $p_{\rm{opt}}$ such that vector fields expressed in new coordinates becomes 
$Y=L+R_{p_{\rm{opt}}}+T_{p_{\rm{opt}}}$, where $R_{p_{\rm{opt}}}$ is a polynomial of degree $p_{\rm{opt}}$ and 
$T_{p_{\rm{opt}}}=O(|x|^{p_{\rm{opt}}+1})$ is analytic and exponentially small i.e. for $0<\epsilon$ small enough we have the estimate
\begin{align*}
  \sup_{||x||\leq \varepsilon}||T_{p_{\rm{opt}}}(x)||\leq M_0 \varepsilon^2 e^{-w/\varepsilon}.
\end{align*}
Moreover $R_{p_{\rm{opt}}}(e^{t\bar{L}}x)=e^{t\bar{L}}R_{p_{\rm{opt}}}(x)$, where $\bar{L}=Q^{-1}N^*Q$.
\end{theorem}

Remark that the cases where $N=y\p{x}$ and $N=(y\p{x}+z\p{y})$ have already been treated in \cite{MR2130546}.
This theorem provides a generalization and a geometric explanation using representation theory of $\slt$ of
these examples. Considering the results in \cite{MR1885678} and \cite{MR2199230} one could wonder whether or
not the given normal form actually converges, when $X$ is analytic (i.e. $\alpha=0$). We think however that,
in general, this is not the case.

In section \ref{sec:hist} of this article we repeat some results of the framework created in
\cite{pre05797622}, in order to be self-contained.  In section \ref{sec:repr} we state some results on the
representation theory of $\slt$. In section \ref{sec:construct} we prove some propositions that lead to the
main results, stated as theorem \ref{theo:mainres} and theorem \ref{theo:mainres2} above and proven in section \ref{sec:mainres}.

\section{Background and notation}
\label{sec:hist}
We recall some standard preliminaries about the used normal form procedure. We follow the outline of \cite{pre05797622}, although
similar ideas appear in \cite{MR923885} and \cite{MR2130546}.
\subsection{Setting}
Let $X=N+R$ be a local formal vector field in the neighbourhood of the origin, $N$ its linear part and $R$ the
part or order $\geq 2$. We will look for a coordinate transform $\Phi^{-1}=I+U$, $U$ of order $\geq 2$, such
that the pullback $\Phi_*(X)=X^\prime=N+R^\prime$. A minor calculation shows that 
\begin{align}
\notag&\Phi_*(X)=N+R^\prime\\
\notag\Leftrightarrow & X \circ \Phi^{-1}=D\Phi^{-1}.X^\prime\\
\notag\Leftrightarrow & (S+R)\circ(I+U)=D(I+U).(N+R^\prime)\\
\label{eq:1}\Leftrightarrow & R^\prime+[U,N]=R(I+U)-DU.R^\prime
\end{align}

Now we are going to determine the terms of order $\delta$ for the formal series $U=U_2+U_3+U_4+\ldots$ and $R^\prime=R_2^\prime+R_3^\prime+R_4^\prime+\ldots$
by induction. Therefore suppose that we already know $U_2$, $\ldots$, $U_{\delta-1}$ and $R_2^\prime$, $\ldots$,  $R_{\delta-1}^\prime$. We take the projection of
the terms of order $\delta$ in ($\ref{eq:1}$) and obtain: 
\begin{align*}
R^\prime_\delta+[U_\delta,N]=RHS_\delta(U_{<\delta},R^{\prime}_{<\delta}).
\end{align*}
where $RHS_\delta$ stands for the projection of order $\delta$ of the right hand side of ($\ref{eq:1}$) and depends only on $U_l, R^{\prime}_l$ with index strictly smaller than $\delta$. Therefore it is natural to introduce the Lie-operator
\begin{align*}
d_{0,\delta}:\mathcal{V}_\delta\longrightarrow \mathcal{V}_\delta:U\mapsto [U,N];
\end{align*}
where
\begin{align*}
\mathcal{V}_\delta&=\left\{\sum_{i=1}^nP_i\p{x_i}|P_i\in \mathcal{P}_{\delta+1}\right\},\\
\mathcal{P}_\delta&=\left\{P|\,\text{$P$ is a homogeneous polynomial of degree $\delta$}\right\},
\end{align*}
and decompose the space $\mathcal{V}_\delta$ of vector fields of degree $\delta$ as $\mathcal{V}_\delta=\Im(d_{0,\delta})\oplus \mathcal{W}_\delta$,
where $\mathcal{W}_\delta$ is a particular choice of a complementary space that is induced by an inner product. This is explained in detail in the
next section. Remark that we will sometimes drop the $\delta$ in the notation whenever no confusion is possible. 

\subsection{The choice of the complementary subspaces $\mathcal{W}_\delta$}

In order to define the complementary spaces $\mathcal{W}_\delta$ we need the adjoint action of $d_0$ with respect to an inner product. Therefore we have the following
\begin{definition}
We define an inner product on $\mathcal{P}_\delta$, the space of polynomials of degree $\delta$ as 
\begin{align*}
%\label{eq:inp}
\left<\sum_{|\alpha|=\delta}a_{\alpha}x^{\alpha},\sum_{|\beta|=\delta}b_{\beta}x^{\beta}\right>=\sum_{|\alpha|=\delta}a_{\alpha}\bar{b}_{\alpha}\frac{\alpha!}{|\alpha|!}.
\end{align*}
This induces an inner product on the space $\mathcal{V}_{\delta-1}$ of vector fields of degree $\delta-1$ as follows:
\begin{align}
\label{eq:inpvv}
\left<\sum_{k=1}^{n} V_k \p{x_k},\sum_{k=1}^{n} W_k \p{x_k}\right>=
\sum_{k=1}^{n} \left<V_k,W_k\right>_{\delta},
\end{align}
where the $V_k,W_k$ are elements of $\mathcal{P}_\delta$.
\end{definition}

Now we define $d_0^*$ as the adjoint action of $d_0$ with respect to the above inner product. We repeat that $d_0^*$ is defined as the unique linear
map satisfying $\left<d_0^*(V),W\right>=\left<V,d_0(W)\right>$, for all $V,\,W \in \mathcal{V}_\delta$ . From linear algebra we know that:

\begin{enumerate}
\item The operators $\Box_\delta=d_0d_0^*$ are self-adjoint.
\item The operators $\Box_\delta$ are diagonizable.
\item The operators $\Box_\delta$ have real positive eigenvalues.
\item $\mathcal{V}_\delta=\Ker(\Box_\delta)\oplus\Im(\Box_\delta)=\Ker(d_0^*)\oplus\Im(d_0)$
\end{enumerate}

We will from now on choose the complementary subspace $W_\delta$ as $\Ker(d_0^*)=\Ker(\Box_\delta)$.

We recall from \cite{pre05797622} a nice way to calculate the adjoint operator $d_0^*$. Let us first define the isomorphism: 
$$\phi:\mathcal{V}_{\delta}\longrightarrow \mathcal{P}_{\delta+1}^n:\sum_{k=1}^nV_k\p{x_k}\mapsto(V_1,V_2,\ldots, V_n)$$
\begin{lemma}[\cite{pre05797622}, p.691]
\label{lem:adjoint} 
Suppose that $V=\sum_{k=1}^nV_k\p{x_k}\in\mathcal{V}_{\delta}$. Then we have 
\begin{align*}
\phi(d_0^*(V))=\left(
\begin{array}{cccc}
N^*-\left(\frac{\partial N_1}{\partial x_1}\right)^*& -\left(\frac{\partial N_2}{\partial x_1}\right)^*&\ldots &-\left(\frac{\partial N_n}{\partial x_1}\right)^*\\
-\left(\frac{\partial N_1}{\partial x_2}\right)^*&N^*-\left(\frac{\partial N_2}{\partial x_2}\right)^*&\ldots &-\left(\frac{\partial N_n}{\partial x_2}\right)^*\\
\vdots &&&\vdots\\
-\left(\frac{\partial N_1}{\partial x_n}\right)^*&\ldots&-\left(\frac{\partial N_{n-1}}{\partial x_n}\right)^*&N^*-\left(\frac{\partial N_n}{\partial x_n}\right)^*
\end{array}\right)
\left(
\begin{array}{c}
  V_1\vphantom{\left(\frac{\partial N_1}{\partial x_1}\right)^*}\\
\vdots\vphantom{\left(\frac{\partial N_1}{\partial x_1}\right)^*}\\
\vdots\\
V_n\vphantom{\left(\frac{\partial N_1}{\partial x_1}\right)^*}
\end{array}\right).
\end{align*}
\end{lemma}

We give an example in the case that $n=2$ and $N=x_2\p{x_1}$. In this case we have
\begin{align*}
\phi(d_0^*(V_1\p{x_1}+V_2\p{x_2}))=\left(
\begin{array}{cc}
N^*& 0\\
-1&N^*
\end{array}\right)
\left(
\begin{array}{c}
V_1\\
V_2
\end{array}\right).
\end{align*}
Hence $d_0^*(V_1\p{x_1}+V_2\p{x_2})=N^*V_1\p{x_1}+(-V_1+N^*V_2)\p{x_2}$.

\subsection{Resonant terms and small denominators}
When solving equation $(\ref{eq:1})$, we decompose $RHS_\delta=Q_\delta\oplus T_\delta$, where $Q_\delta\in \Ker(\Box_\delta)$  and $T_\delta\in \Im(\Box_\delta)=\Im(d_0)$. Now let $\Lambda$ be the set of eigenvalues with multiplicity of the operator $\Box_\delta$ and $\Lambda^*$ be the set of nonzero eigenvalues. Since $\Box_\delta$ is diagonizable, it is possible to decompose $T_\delta$ in a base of eigenvectors of $\Box_{\delta}$. More precisely:
\begin{align*}
T_\delta=\Box_\delta(V_\delta)=
\sum_{\lambda\in \Lambda^*}\Box_\delta(V_{\delta,\lambda})=
\sum_{\lambda\in \Lambda^*}\lambda V_{\delta,\lambda}.
\end{align*}
If we define $W_{\delta,\lambda}=d_0^*(V_{\delta,\lambda})$ and $W_\delta=\sum_{\lambda \in \Lambda^*}W_{\delta,\lambda}$, then

\begin{align*}
d_0(W_\delta)=d_0(\sum_{\lambda \in \Lambda^*}W_{\delta,\lambda})=
\sum_{\lambda \in \Lambda^*}\Box_\delta(V_{\delta,\lambda})=
\sum_{\lambda \in \Lambda^*}\lambda V_{\delta,\lambda}=T_\delta.
\end{align*}

Moreover since 
\begin{align*}
\left<W_{\delta,\lambda},W_{\delta,\lambda}\right>=\left<d_0^*(V_{\delta,\lambda}),d_0^*(V_{\delta,\lambda})\right>=
\left<V_{\delta,\lambda},\Box_\delta(V_{\delta,\lambda})\right>=\lambda\left<V_{\delta,\lambda},V_{\delta,\lambda}\right>,
\end{align*}
it follows that we have the estimates:
\begin{align*}
||T_\delta||^2=||d_0(W_\delta)||^2=\sum_{\lambda \in \Lambda^*}\lambda^2 ||V_{\delta,\lambda}||^2=
\sum_{\lambda \in \Lambda^*}\lambda ||W_{\delta,\lambda}||^2\geq
(\min_{\lambda \in \Lambda^*}(\sqrt{\lambda})||W_\delta||)^2.
\end{align*}
This estimate makes clear that the $\lambda$'s will play the role of the small denominators. We explain now what we mean by `$S$ satisfies a diophantine condition'. Therefore, following \cite{pre05797622}, p .675, we introduce the numbers $a_\delta= \min_{\lambda \in \Lambda^*}(\sqrt{\lambda})$ and define the numbers $\eta_\delta$, for $\delta\geq 0$, recursively by (let $\eta_0=1$)
\begin{align*}
a_\delta\eta_\delta=\max_{\delta_1+\ldots+\delta_r=\delta}\eta_{\delta_1}\ldots\eta_{\delta_r},
\end{align*}
where the maximum is taken over the set where at least two of the $\delta_i$'s  are strictly positive.
We say that $S$ satisfies a diophantine condition if the $\eta_\delta\leq cM^\delta$ for certain positive constants $c,M$.

We will say that `$S$ satisfies a Siegel condition of order $(\tau,\gamma)$' or more shortly `of order $\tau$' whenever we have the estimates:
\begin{align*}
a_\delta \geq\frac{\gamma}{\delta^\tau}
\end{align*}
for a certain $\gamma>0$ and $\tau \geq 0$.
These conditions are important because of the following three theorems:
\begin{theorem}[\cite{pre05797622}, Theorem 5.6 p.676 and Remark 6.7 on p.686]
\label{thm:ann}
Suppose that $X=N+R$ is a formal Gevrey-$\alpha$ vector field that is formally linearizable to its linear part $N$. Suppose that $N$ satisfies a diophantine condition, then $X$ is Gevrey-$\alpha$ linearizable.
\end{theorem}

\begin{theorem}[\cite{pre05797622}, Theorem 6.2 p.683 and Remark 6.7 on p.686]
\label{thm:gev}
Suppose that $X=N+R$ is a Gevrey-$\alpha$ vector field that has a formal normal form $X^\prime=N+R^\prime$ by means of the
procedure explained in this section, and suppose that the linear part $N$ of $X$ satisfies a Siegel condition of order
$\tau$, then $X^\prime$ and $\Phi$ are formal power series of type Gevrey-($1+\tau+\alpha$).
\end{theorem}

\begin{theorem}[\cite{MR2130546}, Corollary 1.9 p.7 and \cite{pre05797622}, Theorem 6.11 p.689]
  \label{thm:optord}
Suppose that $X$ is an analytic vector field in a neighbourhood of $0$ of $\C^m$ that vanishes at the origin. This is
\begin{align*}
  X=L+\sum_{\delta\geq 2}f_{\delta},
\end{align*}
where $L$ is linear, $f_\delta$ is a bounded $\delta$-linear symmetric form and
\begin{align*}
  ||f_\delta(x_1,\ldots,x_\delta)||\leq \frac{c||x_1||\ldots ||x_\delta||}{\rho^\delta},
\end{align*}
with $c,\rho>0$ independent of $\delta$.
Suppose that $Q$ is an invertible linear transformation with norm $||Q||=\sup_{||x||=1}||Q(x)||$ for which $N=QLQ^{-1}$ satisfies
the Siegel condition of order $(\tau,\gamma)$.
Let $\nu=\sup\frac{e^2p!}{p^{p+1/2}e^{-p}}$, $a=1/\gamma$,
\begin{align*} 
  C&=\frac{\sqrt{m}}{\rho^2}\left\{ \left( \frac{5}{2}m+2 \right)ac ||Q||^2.||Q^{-1}||^2+3\rho ||Q||.||Q^{-1}|| \right\}\\
M_{\tau}&=\frac{10}{9}\left\{ \left( \nu\sqrt{\frac{27}{8e}} \right)^{1+\tau}+(2e)^{2+2\tau} \right\},
\end{align*}
$b=\frac{1}{1+\tau}$, $w=\frac{1}{eC^b}$ and $p_{\rm{opt}}=\left[ \frac{1}{eC^b} \right]$. There exists a polynomial coordinate transformation
$\id+u_{p_{\rm{opt}}}$
of degree at most $p_{\rm{opt}}$ such that vector fields expressed in new coordinates becomes 
$Y=L+R_{p_{\rm{opt}}}+T_{p_{\rm{opt}}}$, where $R_{p_{\rm{opt}}}$ is a polynomial of degree $p_{\rm{opt}}$ and 
$T_{p_{\rm{opt}}}=O(|x|^{p_{\rm{opt}}+1})$ is analytic and exponentially small i.e. for $0<\epsilon$ small enough we have the estimate
\begin{align*}
  \sup_{||x||\leq \varepsilon}||T_{p_{\rm{opt}}}(x)||\leq M_\tau \varepsilon^2 e^{-w/\varepsilon^b}.
\end{align*}
Moreover $R_{p_{\rm{opt}}}(e^{t\bar{L}}x)=e^{t\bar{L}}R_{p_{\rm{opt}}}(x)$, where $\bar{L}=Q^{-1}N^*Q$.
\end{theorem}

\begin{remark}
  
  If $Q$ is a unitary transformation, then the above estimates hold with $||Q||$, $||Q^{-1}||$ replaced by 1; moreover then
  also $\bar{L}=N^*$ holds.
\end{remark}

\section{Representations of $\slt$}
\label{sec:repr}
We briefly recall the definition of a Lie algebra, a representation of a Lie algebra and some related algebraic concepts.
\begin{definition}
A Lie algebra $(\mathfrak{g},[\,,])$ is a vector space $\mathfrak{g}$ provided with a multiplication $[\,,]:\mathfrak{g}\times \mathfrak{g} \mapsto \mathfrak{g}:(x,y)\mapsto [x,y]$ that satisfies the relations
\begin{itemize}
\item $[g_1,g_2]=-[g_2,g_1]$,
\item $[g_1,[g_2,g_3]]+[g_2,[g_3,g_1]]+[g_3,[g_1,g_2]]=0.$
\end{itemize}
We list the following concepts:
\begin{enumerate}
\item A Lie algebra $\mathfrak{g}$ is called simple iff $[\mathfrak{g},\mathfrak{g}]=\mathfrak{g}$.
\item A lie algebra homomorphism is a linear map $L:\mathfrak{g}\longrightarrow \mathfrak{h}$ between two Lie algebra's preserving the product structure : $L([g_1,g_2])=[L(g_1),L(g_2)]$.
\item A Lie algebra representation of $\mathfrak{g}$ is a Lie algebra homomorphism $L:\mathfrak{g}\longrightarrow gl(V)$, where $V$ is a vector space and $gl(V)$ is the group of linear transformations from $V$ to $V$.
\item A Lie algebra representation $L:\mathfrak{g}\longrightarrow gl(V)$ is irreducible, iff there exist no subspace $W$ different from $V$ or $\{0\}$ such that $L(g)(w)\in W$, for every $w\in W$ and every $g\in \mathfrak{g}$. A subspace $W$ with this property defines a subrepresentation.
\end{enumerate}
\end{definition}
We will need one of the key results of representations of simple Lie algebra's. A proof can be found e.g. in \cite{MR0323842}.
\begin{theorem}
\label{theo:simp}
Every finite dimensional representation of a simple Lie algebra $\mathfrak{g}$ can be written as a direct sum of irreducible representations of $\mathfrak{g}$.
\end{theorem}

We now recall some basic facts of the representations of the simple Lie algebra $\slt$. Let us first recall the definition.
\begin{definition}
We define the Lie algebra $\slt$ as the subalgebra of $gl_2(\C)$ of matrices with trace $0$. It is generated by the matrices
\begin{align*}
N=\left(
\begin{array}{cc}
0 & 1\\
0 & 0
\end{array}\right),
M=\left(
\begin{array}{cc}
0 & 0\\
1 & 0
\end{array}\right),
H=\left(
\begin{array}{cc}
1 & 0\\
0 & -1
\end{array}\right).
\end{align*}
\end{definition}
\begin{remark}
\label{rem:liealg}
Any Lie algebra generated by three elements $N,M,H$ and subject to the relations
\begin{align*}
[H,N]=2N,\,
[H,M]=-2M,\,
[N,M]=H,
\end{align*}
is isomorphic with $\slt$. Moreover it is now clear that $\slt$ is a simple Lie algebra.
\end{remark}
The following theorem is well-known: a proof can be found e.g. in \cite{MR0323842}.
\begin{theorem}
\label{theo:irep}
For every $n$ the representation of $\slt$ defined by
\begin{align*}\left(
\begin{array}{cc}
0 & 1\\
0 & 0
\end{array}\right)&\mapsto \widetilde{N}_n=\left(
\begin{array}{cccccc}
0 & n &0&0& \ldots &0\\
0 & 0& n-1&0&\ldots&0\\
\vdots&\vdots&\vdots&\vdots&\vdots&\vdots\\
0 & 0&\ldots&0&2&0\\
0 & 0&\ldots&0&0&1\\
0 & 0&\ldots&0&0&0\\
\end{array}\right)\\
\left(
\begin{array}{cc}
0 & 0\\
1 & 0
\end{array}\right)&\mapsto \widetilde{M}_n=\left(
\begin{array}{cccccc}
0 & 0 &0&0& \ldots &0\\
1 & 0& 0&0&\ldots&0\\
0 & 2& 0&0&\ldots&0\\
\vdots&\vdots&\vdots&\vdots&\vdots&\vdots\\
0 & 0&\ldots&n-1&0&0\\
0 & 0&\ldots&0&n&0\\
\end{array}\right)\\
\left(
\begin{array}{cc}
1 & 0\\
0 & -1
\end{array}\right)&\mapsto \widetilde{H}_n=\left(
\begin{array}{cccccc}
n & 0 &0&0& \ldots &0\\
0 & n-2& 0&0&\ldots&0\\
\vdots&\vdots&\vdots&\vdots&\vdots&\vdots\\
0 & 0&\ldots&0&-(n-2)&0\\
0 & 0&\ldots&0&0&-n\\
\end{array}\right)
\end{align*}
and acting on $\C^{n+1}$ is irreducible.
Moreover, any other irreducible representation of $\slt$ is isomorphic with one of these representations.
\end{theorem}

\section{Construction of some particular $\slt$ representations}
\label{sec:construct}

In this section we focus on the construction of some particular $\slt$ representations. 
In order to make the computations a bit more transparent, we use the correspondence between matrices and vector fields by a bijection $\phi:\sum_i\sum_j a_{i,j}x_j\p{x_i}\mapsto A=(a_{ij})$. Now suppose that we have two vector fields  $A_v=\sum_{j=1}^n\sum_{i=1}^n a_{ij}x_i\p{x_j}$ and $B_v=\sum_{j=1}^n\sum_{i=1}^n b_{ij}x_i\p{x_j}$ with corresponding matrices $A$ and $B$, then the Lie bracket transforms as $\phi([A_v,B_v])=AB-BA$.

We start now with the construction. Therefore we define the following vector fields:
\begin{align*}
N_n&:=\alpha_1 x_2\p{x_1}+\alpha_2 x_3\p{x_2}+\ldots +\alpha_n x_{n+1}\p{x_n}\\
M_n&=\alpha_1 x_1\p{x_2}+\alpha_2 x_2\p{x_3}+\ldots +\alpha_n x_{n}\p{x_{n+1}}\\
H_n&:=[N_n,M_n].
\end{align*}
It is important to note that $M_n$ is the adjoint of $N_n$ with respect to the inner product (\ref{eq:inpvv}). We will use the same notation $N_n$, $M_n$ and $H_n$ for the associated matrices and drop the index $n$ where no confusion is possible.
We want to choose the coefficients $\alpha_1,\,\ldots ,\,\alpha_n$ in such a way that they are non-zero and that the triple $N$, $M$, $H$ is isomorphic to the Lie algebra $\slt$. Therefore it is sufficient to ensure that the relations
described in remark \ref{rem:liealg} are satisfied. The third relation is automatic from the construction. We focus on the first relation. In matrix notation this relation becomes $HN-NH-2N=0$. Now remark that 
\begin{align*}
N=\left(\begin{array}{ccccc}
0 & \alpha_1 &0& \ldots &0\\
0 & 0& \alpha_2&\ldots&0\\
\vdots&\vdots&\vdots&\vdots&\vdots\\
0 & 0&\ldots&0&\alpha_n\\
0 & 0&\ldots&0&0\\
\end{array}\right),\,
M=\left(\begin{array}{ccccc}
0 & 0 &0& \ldots &0\\
\alpha_1& 0&0&\ldots&0\\
0& \alpha_2&0&\ldots&0\\
\vdots&\vdots&\vdots&\vdots&\vdots\\
0 & 0&\ldots&\alpha_n&0\\
\end{array}\right),\\
H=\left(\begin{array}{cccccc}
\alpha_1^2 & 0 &0& \ldots &0\\
0 &  \alpha_2^2-\alpha_1^2&0&\ldots&0\\
0&0&\alpha_3^2-\alpha_2^2&\ldots&0\\
\vdots&\vdots&\vdots&\vdots&\vdots\\
0&0&\ldots&\alpha_{n}^2-\alpha_{n-1}^2&0\\
0 & 0&\ldots&0&-\alpha_n^2\\
\end{array}\right).
\end{align*}
Hence this relation becomes
\begin{align*}
\left(\begin{array}{cccccc}
0 & b_1&0 &0& \ldots &0\\
0 & 0& b_2&0&\ldots&0\\
\vdots&\vdots&\vdots&\vdots&\vdots&\vdots\\
0 & 0&\ldots&0&0&b_n\\
0 & 0&\ldots&0&0&0
\end{array}\right)=0,
\end{align*}
where $b_i=\alpha_i(-\alpha_{i-1}^2+2\alpha_i^2-\alpha_{i+1}^2)-2\alpha_i$ for $2\leq i\leq n-1$, $b_1=\alpha_1(2\alpha_1^2-\alpha_2^2)-2\alpha_1$ and $b_n=\alpha_n(-\alpha_{n-1}^2+2\alpha_n^2)-2\alpha_n$;
and we need to solve the equations 
\begin{align*}
\left\{\begin{array}{l}
\alpha_1(2\alpha_1^2-\alpha_2^2)=2\alpha_1\\
\alpha_2(-\alpha_1^2+2\alpha_2^2-\alpha_3^2)=2\alpha_2\\
\,\,\,\,\vdots\\
\alpha_n(-\alpha_{n-1}^2+2\alpha_n^2)=2\alpha_n.
\end{array}\right.
\end{align*}
Since we suppose that none of the $\alpha_i$ vanish, this simplifies to 
\begin{align*}
\left(\begin{array}{cccccc}
2&-1&0&0&\ldots &0\\
-1&2&-1&0&\ldots&0\\
\vdots&\vdots&\vdots&\vdots&\vdots&\vdots\\
0&\ldots&0&-1&2&-1\\
0&\ldots&0&0&-1&2
\end{array}\right)
\left(\begin{array}{c}
\alpha_1^2\\
\alpha_2^2\\
\vdots\\
\alpha_{n-1}^2\\
\alpha_n^2
\end{array}\right)
=\left(\begin{array}{c}
2\\
2\\
\vdots\\
2
\end{array}\right).
\end{align*}
One can verify that a solution is given by $\alpha_i^2=i(n+1-i)$ for $1\leq i\leq n$. We choose the positive solutions and put $\alpha_i=\sqrt{i(n+1-i)}$. Then it is readily verified (repeat the above calculations) that also the second relation $[H,M]=-2M$ from remark \ref{rem:liealg} is satisfied. We have now proven the
\begin{lemma}
\label{lem:reprth1}
Let $n\in\N$ and define $N_n=\sum_{i=1}^n\sqrt{i(n+1-i)} x_{i+1}\p{x_i}$, then the triple $N_n$, $M_n:=N_n^*$, $H=[N_n,M_n]$ defines a Lie-algebra isomorphic to $\slt$.
\end{lemma}

We are now ready to show that 
\begin{lemma}
\label{lem:reprth}
Let $\delta\in \N\setminus\{0\}$. For a given $N_n=\sum_{i=1}^n\sqrt{i(n+1-i)} x_{i+1}\p{x_i}$ the associated triple $d_0$, $d_0^*$, $D=[d_0,d_0^*]$ defines an $\slt$ representation. Here is $d_0$ the Lie operator $U\mapsto [N_n,U]$ acting on $\mathcal{V}_\delta$ and $d_0^*$ its adjoint.
\end{lemma}
\begin{proof}
Put $\alpha_i=\sqrt{i(n+1-i)}$ and $I$ the identity operator. The Lie operator acting on vector fields $d_0(\sum_{i=1}^{n+1}V_i\p{x_i})$ can be expressed using matrix notation as
\begin{align*}
\left(\begin{array}{cccccc}
N&-\alpha_1I&0&0\ldots &0\\
0&N&-\alpha_2I&0\ldots &0\\
\vdots&\vdots&\vdots&\vdots&\vdots&\vdots\\
0&\ldots&0&0&N&-\alpha_nI\\
0&\ldots&0&0&0&N\\
\end{array}\right)
\left(\begin{array}{c}
V_1\\
V_2\\
\vdots\\
V_{n}\\
V_{n+1}
\end{array}\right);
\end{align*}
and its adjoint $d_0^*(\sum_{i=1}^{n+1}V_i\p{x_i})$ as (see also lemma \ref{lem:adjoint} and remember that $M=N^*$)
\begin{align*}
\left(\begin{array}{cccccc}
M&0&0&0\ldots &0\\
-\alpha_1I&M&0&0\ldots &0\\
\vdots&\vdots&\vdots&\vdots&\vdots&\vdots\\
0&\ldots&0&-\alpha_{n-1}I&M&0\\
0&\ldots&0&0&-\alpha_nI&M\\
\end{array}\right)
\left(\begin{array}{c}
V_1\\
V_2\\
\vdots\\
V_{n}\\
V_{n+1}
\end{array}\right).
\end{align*}
Hence it is readily verified that the commutator $D=[d_0,d_0^*]=d_0d_0^*-d_0^*d_0$ can be expressed as
\begin{align*}
\left(\begin{array}{cccccc}
H+\alpha_1^2I & 0 &0& \ldots &0\\
0 & H+(\alpha_2^2-\alpha_1^2)I&0&\ldots&0\\
0&0&H+(\alpha_3^2-\alpha_2^2)I&\ldots&0\\
\vdots&\vdots&\vdots&\vdots&\vdots\\
0&0&\ldots&H+(\alpha_{n}^2-\alpha_{n-1}^2)I&0\\
0 & 0&\ldots&0&H-\alpha_n^2I\\
\end{array}\right),
\end{align*}
where $H=[N,M]$.
Now the commutator $[D,d_0]$ simplifies as
\begin{align*}
\left(\begin{array}{cccccc}
a_1&b_1&0&0\ldots&0\\
0&a_2&b_2&0&0\ldots&0\\
\vdots&\vdots&\vdots&\vdots&\vdots&\vdots\\
0&0&0&0\ldots& a_n&b_{n}\\
0&0&0&0\ldots& 0&a_{n+1}
\end{array}\right),
\end{align*}
with 
\begin{align*}
a_i&=HN-NH+ (\alpha_i^2-\alpha_{i-1}^2)(NI-IN)=[H,N]=2N,\\
b_i&=-\alpha_i(H+\alpha_i^2I-\alpha_{i-1}^2I)+\alpha_i(H+\alpha_{i+1}^2I-\alpha_i^2I)=\\
&=-\alpha_i(-\alpha_{i+1}^2+2\alpha_i^2-\alpha_{i-1}^2)I\\
&=-2\alpha_iI;
\end{align*}
where we have put $\alpha_0=0$ and $\alpha_{n+1}=0$ in the above calculation. We also used the fact that the triple $N$, $M$, $H$ defines a Lie algebra isomorphic to $\slt$.
Hence $[D,d_0]=2d_0$.\\
Making analogous calculations, one verifies that also $[D,d_0^*]=-2d_0^*$.
\end{proof}
As a corollary of this lemma we can consider the case of multiple nilpotent blocks as follows. Remark that we allow zero blocks (i.e. $k_i=0$ for some $i$).
\begin{proposition}
\label{prop:mainprop} Let $k_1$, $k_2$, $\ldots$, $k_n$ be natural numbers and let $x^i$ be a $k_i$-dimensional variable $(x^i_1,\ldots,x^i_{k_i+1})$, for $1\leq i\leq n$.
Let $N=N_{k_1}(x^1)+\ldots +N_{k_n}(x^n)$, where 
\begin{align*}
N_{k_j}(x^j)=\sum_{i=1}^{k_j} \sqrt{i(n-i+1)}x^j_{i+1}\p{x^j_i},\qquad N_0=0.
\end{align*}
Then the triple $N$, $M:=N^*$, $H=[N,M]$ defines a Lie algebra isomorphic to $\slt$. Moreover let $d_0$ be the associated Lie operator, then also the triple $d_0$, $d_0^*$, $D=[d_0,d_0^*]$ defines a Lie algebra isomorphic to $\slt$.
\end{proposition}
\begin{proof}
Use the concept of a direct sum, lemma \ref{lem:reprth1} and lemma \ref{lem:reprth}.
\end{proof}

\section{Proof of the main theorems}
\label{sec:mainres}
This is rather a summary of all the foregoing.
From linear algebra we know that, up to a linear change of variables, it is no restriction to start with a vector field $X=N+R$ where $N$ is as in proposition \ref{prop:mainprop}. Let now $d_0$ be the associated Lie operator. Let $\delta\in\N\setminus\{0\}$. We are interested in the calculation of eigenvalues of the associated operator $\Box_\delta=d_0d_0^*$ acting on $\mathcal{V}_\delta$.
According to proposition \ref{prop:mainprop}, we know that the triple $d_0$, $d_0^*$ and $D=[d_0,d_0^*]$ defines a Lie algebra isomorphic to $\slt$. It follows, using theorem \ref{theo:simp}, that the associated representation acting  on $\mathcal{V}_\delta$ can be split in a direct sum of irreducible representations. Hence, up to a linear coordinate transform $\varphi$ (acting on the space $\mathcal{V}_\delta$), we can suppose that we are dealing with a representation of the form
\begin{align*}
&N=\widetilde{N}_1\oplus \widetilde{N}_2\oplus \ldots\oplus \widetilde{N}_l,\\ 
&M=\widetilde{M}_1\oplus \widetilde{M}_2\oplus \ldots\oplus \widetilde{N}_l,\\
&H=\widetilde{H}_1\oplus \widetilde{H}_2\oplus \ldots\oplus \widetilde{H}_l;
\end{align*}
 where $\widetilde{N}_i$, $\widetilde{M}_i$ and $\widetilde{H}_i$ are as in theorem \ref{theo:irep}. Hence $\varphi$ transforms
 the operator $\Box_\delta=d_0d_0^*$ into $NM$. The nonzero eigenvalues of the operator
 $NM=\widetilde{N}_1\widetilde{M}_1\oplus\ldots\oplus \widetilde{N}_l\widetilde{M}_l$ are positive integers because each
 $\widetilde{N}_i\widetilde{M}_i$ is a diagonal matrix containing integers on the diagonal. Hence the same is true for the
 operator $\Box_\delta$. Now using theorem \ref{thm:gev} with $\tau=0$ and $\gamma=1$ (or theorem \ref{thm:ann} in case the
 vector field is formally linearizable) finishes the proof theorem \ref{theo:mainres}.  
 Similarly, using theorem \ref{thm:optord} with $\tau=0$ and $\gamma=1$  finishes the proof theorem \ref{theo:mainres2}.

\section{Acknowledgement}
This research was initiated by the second author during a visit at the university of Nice, where L. Stolovitch pointed him to this problem. We also wish to thank M. Van den Bergh for providing some useful ideas on the theory of representations. 

\bibliographystyle{plain}
\bibliography{bibdatabase}
\end{document}